\documentclass[12pt,reqno]{amsart}

\usepackage{amsmath}
\usepackage{amssymb}
\usepackage{amsfonts}
\usepackage{setspace}
\usepackage{mathrsfs}
\usepackage{version}
\usepackage{url}

\DeclareMathAlphabet{\curly}{U}{rsfs}{m}{n}  

\newtheorem{theorem}{Theorem}[section]
\newtheorem{lemma}[theorem]{Lemma}
\newtheorem{proposition}[theorem]{Proposition}
\newtheorem{corollary}[theorem]{Corollary}

\theoremstyle{definition}

\newtheorem{question}{Question}

\renewcommand{\leq}{\leqslant}
\renewcommand{\geq}{\geqslant}

\newcommand\Pois{\operatorname{Pois}}

\def\Z{\mathbb{Z}}
\def\E{\mathbb{E}}
\def\P{\mathbb{P}}

\def\N{\mathbb{N}}
\def\b{\mathbf{b}}
\def\X{\mathbf{X}}
\def\Y{\mathbf{Y}}
\def\c{\mathbf{c}}

\def\eps{\varepsilon}
\renewcommand{\a}{\alpha}
\renewcommand{\b}{\beta}
\renewcommand{\d}{\mathbf{d}}
\newcommand{\ba}{\mathbf{a}}
\newcommand{\bb}{\mathbf{b}}
\newcommand{\DD}{\curly D}

\makeatletter
\renewcommand{\pmod}[1]{\allowbreak\mkern7mu({\operator@font mod}\,\,#1)}
\makeatother

\newcommand{\be}{\begin{equation}}
\newcommand{\ee}{\end{equation}}

\newcommand{\ssum}[1]{\sum_{\substack{#1}}}  
\renewcommand{\(}{\left(}
\renewcommand{\)}{\right)}

\newcommand{\pfrac}[2]{\left(\frac{#1}{#2}\right)}  
\newcommand{\order}{\asymp}      
\newcommand{\fl}[1]{{\ensuremath{\left\lfloor {#1} \right\rfloor}}}

\renewcommand{\le}{\leqslant}
\renewcommand{\ge}{\geqslant}

\newcommand{\cS}{\mathcal{S}}  
\newcommand{\cA}{\mathcal{A}}
\newcommand{\cC}{\mathcal{C}}  
\newcommand{\sL}{\mathscr{L}}

\newcommand{\deq}{\overset{d}=}  

\numberwithin{equation}{section}

\begin{document}

\title{Permutations fixing a $k$-set}

\author{Sean Eberhard}
\address{Mathematical Institute\\
Radcliffe Observatory Quarter\\
Woodstock Road\\
Oxford OX2 6GG\\
England }
\email{sean.eberhard@maths.ox.ac.uk}

\author{Kevin Ford}
\address{Department of Mathematics, 1409 West Green Street, University
of Illinois at Urbana-Champaign, Urbana, IL 61801, USA}
\email{ford@math.uiuc.edu}

\author{Ben Green}
\address{Mathematical Institute\\
Radcliffe Observatory Quarter\\
Woodstock Road\\
Oxford OX2 6GG\\
England }
\email{ben.green@maths.ox.ac.uk}

\thanks{BG is supported by ERC Starting Grant number 279438, \emph{Approximate algebraic structure and applications}, and a Simons Investigator Award.}
\thanks{KF is supported by National Science Foundation grants DMS-1201442 and DMS-1501982.}


\begin{abstract}
Let $i(n,k)$ be the proportion of permutations $\pi\in\cS_n$ having an invariant set of size $k$. In this note we adapt arguments of the second author to prove that $i(n,k) \asymp k^{-\delta} (1+\log k)^{-3/2}$ uniformly for $1\leq k\leq n/2$, where $\delta = 1 - \frac{1 + \log \log 2}{\log 2}$. As an application we show that the proportion of $\pi\in\cS_n$ contained in a transitive subgroup not containing $\cA_n$ is at least $n^{-\delta+o(1)}$ if $n$ is even.
\end{abstract}

\maketitle

\section{Introduction and notation}

Let $k,n$ be integers with $1\le k\le n/2$ and select a permutation 
$\pi \in \cS_n$, that is to say a permutation of $\{1,\dots, n\}$, at random. What is $i(n,k)$, the probability that $\pi$ fixes some set of size $k$? Equivalently, what is the probability that the cycle decomposition of $\pi$ contains disjoint cycles with lengths summing to $k$?

Somewhat surprisingly, $i(n,k)$ has only recently been at all well understood in the published literature. The lower bound $\lim_{n\to\infty} i(n,k) \gg \log k/k$ is contained in a paper of Diaconis, Fulman and Guralnick \cite{dfg08}, while the upper bound $i(n,k) \ll k^{-1/100}$ may be found in work of {\L}uczak and Pyber \cite{lp93}. (These authors did not make any special effort to optimise the constant 1/100, but their method does not lead to a sharp bound.) Here and throughout $X\ll Y$ means $X\leq CY$ for some constant $C>0$. The notation $X\order Y$ will be used to mean $X\ll Y$ and $X\gg Y$. In the limit as $n \rightarrow \infty$ with $k$ fixed, a much better bound was very recently obtained by Pemantle, Peres, and Rivin \cite[Theorem 1.7]{PPR}. They prove that $\lim_{n \rightarrow \infty} i(n,k) = k^{-\delta + o(1)}$, where $$\delta = 1 - \frac{1 + \log \log 2}{\log 2}\approx 0.08607.$$ They also note a connection between the problem of estimating $i(n,k)$ and a certain number-theoretic problem, an analogy that will be also be key to our work. The same connection has also been observed by Diaconis and Soundararajan~\cite[page 14]{sound-ams}.

Let us explain the connection with number theory. There is a well known analogy (see, for example, \cite{ABT}) between the cycle decomposition of a 
random permutation and the prime factorisation of a random integer.  Specifically, if $\pi$ is a random 
permutation with cycles of lengths $a_1 \le a_2 \le \dots$, and if $n$ is a random integer with prime 
factors $p_1 < p_2 < \dots$ then one expects both sequences $\log a_1, \log a_2, \dots$ and $\log \log p_1 , 
\log \log p_2 , \dots$ to behave roughly like Poisson processes with intensity $1$. 
(Of course, this does not make sense if taken too literally, since the $a_i$ are all integers, and the $p_i$ 
are all primes, plus we have not specified exactly what we mean by either a ``random permutation'' or a 
``random integer''.) The condition that $a_{i_1} + \dots + a_{i_m} = k$ (that is, that a particular set of cycle 
lengths sum to $k$) is, because the $a_i$ are all integers, equivalent to
$k \leq a_{i_1} + \dots + a_{i_m} < k+1$. Pursuing the analogy between cycles and primes, we may equate 
this with the condition $k \leq \log p_{i_1} + \dots + \log p_{i_m} \leq k+1$, or in other words 
$e^k \leq p_{i_1} \cdots p_{i_m} \leq e^{k+1}$. This then suggests that we might compare $i(n,k)$ 
with $\tilde i(n,k)$, the probability that a random very large integer (selected uniformly from $[e^n, 
e^{n+1})$, say) has a divisor in the range $[e^k, e^{k+1})$.

This last problem has a long history, originating as a problem of Besicovitch~\cite{Bes} in 1934, and was solved (up to a constant factor) by the second author~\cite{ford,ford-2}. In those papers it was shown that $\tilde i(n,k) \asymp k^{-\delta} (1 + \log k)^{-3/2}$ uniformly for $k\le n/2$, where $\delta$ is the constant mentioned above. In this paper we use the same method to prove the same rate of decay for $i(n,k)$.

\begin{theorem}\label{mainthm}
$i(n,k) \asymp k^{-\delta} (1+\log k)^{-3/2}$ uniformly for $1\le k\le n/2$.
\end{theorem}

Since $i(n,n-k)=i(n,k)$, Theorem \ref{mainthm} establishes
the order of $i(n,k)$ for all $n,k$.

Theorem~\ref{mainthm} has implications for a conjecture of Cameron related to random generation of the symmetric group. Cameron conjectured that the proportion of $\pi\in\cS_n$ contained in a transitive subgroup not containing $\cA_n$ tends to zero: this was proved by {\L}uczak and Pyber~\cite{lp93} using their bound $i(n,k)\ll k^{-1/100}$. Cameron further guessed that this proportion might decay as fast as $n^{-1/2+o(1)}$ (see~\cite[Section 5]{lp93}). However Theorem~\ref{mainthm} has the following corollary.

\begin{corollary}\label{trans}
The proportion of $\pi\in\cS_n$ contained in a transitive subgroup not containing $\cA_n$ is $\gg n^{-\delta}(\log n)^{-3/2}$, provided that $n$ is even and greater than $2$.
\end{corollary}
\begin{proof}
By Theorem~\ref{mainthm} the proportion of $\pi\in\cS_n$ fixing a set $B_1$ of size $n/2$ is $\order n^{-\delta}(\log n)^{-3/2}$. Such a permutation $\pi$ must also fix the set $B_2=\{1,\dots,n\}\setminus B_1$, and thus preserve the partition $\{B_1,B_2\}$ of $\{1,\dots,n\}$. Since $|B_1|=|B_2|$, the set of all $\tau$ preserving this partition is a transitive subgroup not containing $\cA_n$.
\end{proof}

We believe that a matching upper bound $O(n^{-\delta}(\log n)^{-3/2})$ holds in Corollary~\ref{trans}, and that for odd $n$ there is an upper bound of the form $O(n^{-\delta'})$ for some $\delta'>\delta$. We intend to return to this problem in a subsequent paper.

Whether or not a permutation $\pi$ has a fixed set of size $k$ depends only on the vector  $\c=(c_1(\pi),c_2(\pi),\ldots,c_k(\pi))$ listing the number of cycles of length $1,2,\ldots,k$, respectively, in $\pi$.
Crucial to our argument is the well known fact (see, e.g., \cite{ABT}) that for \emph{fixed} $k$,  $\c$ has limiting distribution (as $n\to\infty$) equal to $\X_k=(X_1,X_2,\dots,X_k)$, where the $X_i$ are independent and $X_i$ has Poisson distribution with parameter $1/i$ (for short, $X_i \deq \Pois(1/i)$). A simple corollary is that the limit $i(\infty,k) = \lim_{n \rightarrow \infty} i(n,k)$ exists for every $k$. Define, for any finite list $\c=(c_1,c_2,\ldots,c_k)$ of non-negative integers, the quantity
\begin{equation}\label{Lc}
\sL(\c) = \{ m_1 + 2m_2 + \dots +km_k:  0 \leq m_j \leq c_j \; \; 
\mbox{for $j =1,2, \dots, k$}\big\}.
\end{equation}
We immediately obtain that
\begin{equation}\label{fund-inf} i(\infty,k) = \P(k \in \sL(\X_k)).
\end{equation}
This makes it easy to compute $i(\infty,k)$ for small values of $k$. For example we have the extremely well known result (derangements) that
\[ i(\infty,1) = \P(X_1 \geq 1) = 1 - \frac{1}{e} \approx 0.6321,\]
and the less well known fact that
\[ i(\infty,2)  = 1 - \P(X_1 = X_2 = 0) - \P(X_1 = 1, X_2 = 0) = 1- 2e^{-3/2} \approx 0.5537.\]

When $k$ is allowed to grow with $n$, the vector $\c$ is still close to being distributed as  $\X_k$, the total variation distance between the two distributions decaying rapidly as $n/k\to\infty$ \cite{AT92}. This fact is, however, not strong enough for our application. We must establish an approximate analog of \eqref{fund-inf}, showing that $i(n,k)$ has about the same order as  $\P(k \in \sL(\X_k))$, uniformly in $k\le n/2$.

Instead of directly estimating the probability of a single number lying in $\sL(\X_k)$, however, we apply a local-to-global principle
used in \cite{ford,ford-2} to reduce the problem to studying the \emph{size} of $\sL(\X_k)$. We expect a positive proportion of the elements of $\sL(\X_k)$ to lie in the range $[\frac{1}{10}k, 10k]$ (say). The reason for this is that we expect to find $\sim 1$ index $j$ for which $X_j > 0$ in any interval 
$[e^i, e^{i+1}]$. In particular, it is fairly likely that there is some such $j$ with $j > k/10$, 
in which case at least half of the sums $m_1 + 2m_2 + \dots + km_k$ will be $\geq k/10$ (those with $m_j > 0$),
yet at the same time it is reasonably likely that \emph{all} elements of $\mathscr{L}(\X_k)$ are $< 10k$. 
Assuming this heuristic is reasonable, we might expect that 
\begin{equation}\label{heur} i(n,k) \order \P(k \in \sL(\X_k)) \asymp \frac{1}{k}\E|\mathscr{L}(\X_k)|.
\end{equation}
In Section \ref{first-red}, we will show that \eqref{heur} does indeed hold.
The main result of that section is the following.

\begin{proposition}\label{first-reduction}$i(n,k) \order \frac{1}{k} \E |\sL(\X_k)|$ uniformly for $1 \le k\le n/2$.
\end{proposition}

Our main theorem follows immediately from this and the next proposition, whose proof occupies Sections 
\ref{lower} (lower bound) and \ref{upper} (upper bound). Note that in these propositions 
we operate with the sequence $\X_k = (X_1,X_2,\dots,X_k)$ of genuinely independent random variables, which
is independent of $n$.

\begin{proposition}\label{average-result}
$\E |\mathscr{L}(\X_k)| \asymp k^{1-\delta}(1+\log k)^{-3/2}$. 
\end{proposition}

To briefly explain the origin of the exponent $\delta$, we first observe the simple inequalities
\begin{equation}\label{sLup1}
 |\sL(\X_k)| \le \min \( 2^{X_1+\cdots + X_k}, 1+X_1+2X_2+\cdots+kX_k \).
\end{equation}
Assume this is  close to being sharp with reasonably high probability, and condition on $Y=X_1+\cdots+X_k$, the number of
cycles of length at most $k$ in a random permutation.  Following our earlier heuristic, the second term 
on the right side of \eqref{sLup1} is $\asymp k$ most of the time, and so there is a change of behaviour 
around $Y=\frac{\log k}{\log 2}+O(1)$.  Since $Y$ is Poisson with
parameter $\log k+ O(1)$, a short calculation reveals that
$\E \min (2^Y,k) \order k^{1-\delta} (\log k)^{-1/2}$.  We err in the logarithmic term due to the fact that
\eqref{sLup1} is only sharp with probability about $1/\log k$, a fact that is related to order statistics
\cite[Sec.~4]{ford-2}.


Let us finally mention two open questions.

\begin{question}
Is there some constant $C$ such that $i(\infty,k) \sim C k^{-\delta} (\log k)^{-3/2}$?
\end{question}
It would be surprising if this were not the case. 

\begin{question}
Is $i(\infty,k)$ monotonically decreasing in $k$?
\end{question}

Data collected by Britnell and Wildon~\cite{britnell-wildon} shows that this is so at least as far as $i(\infty,30)$, and of course a positive answer is plausible just from the fact that $i(\infty,k) \to 0$. 

%
%
\section{A permutation sieve}\label{sec-perm-sieve}
%
%

As mentioned in the introduction, the asymptotic distribution (as $n \rightarrow \infty$ with $k$ fixed) of the cycle lengths $(c_1(\pi), \dots, c_k(\pi))$ of a random $\pi \in \cS_n$ is that of $\X_k = (X_1,\dots, X_k)$, where the $X_i$ are independent with $X_i \deq \Pois(1/i)$. In the nonasymptotic regime, where $n$ may be as small as $2k$, this property is lost. We do, however, have the following substitute which will suffice for this paper. 

\begin{proposition}\label{sieve2}
 Let $1\le m<n$ and $c_1,\ldots,c_m$ be non-negative integers satisfying $$c_1+2c_2+\cdots +mc_m\le n-m-1.$$ Suppose that $\pi \in \cS_n$ is chosen uniformly at random. Then
 \[ \frac{1}{(2m+2) \prod_{i=1}^m c_i! i^{c_i}} \leq \P(c_1(\pi) = c_1,\dots, c_m(\pi) = c_m) \leq \frac{1}{(m+1) \prod_{i=1}^m c_i! i^{c_i}}.\]
\end{proposition}

We will prove this shortly, but first let us fix some notation. As every permutation $\pi \in \cS_n$ factors uniquely as a product of disjoint cycles, in keeping with the analogy with integers we say that any product of these cycles, including the empty product, is a \emph{factor} or \emph{divisor} of $\pi$. The sets induced by these factors are precisely the invariant sets of $\pi$. We make the following further definitions:

\begin{itemize}
 \item $\mathcal{C}_{k,n}$ is the set of cycles of length $k$ in $\cS_n$;
 \item $|\sigma|$ is the length of any factor $\sigma$ (of some permutation in $\cS_n$);
 \item $\tau|\pi$ means that $\tau$ is an invariant set or divisor of $\pi$.
\end{itemize}

The following lemma is a slight generalization of the well known formula of Cauchy.

\begin{lemma}\label{Cauchy}
Let $1\le m\le n$, and let $c_1,\ldots,c_m$ be non-negative integers with $t=c_1+2c_2+\cdots+mc_m\le n$. Then the number of ways of choosing $c_1+\cdots+c_m$ disjoint cycles consisting of $c_i$ cycles in $\cC_{i,n}$ for $1\le i\le m$ is
 \[
 \frac{n!}{(n-t)!} \prod_{j=1}^m \frac{1}{c_j! j^{c_j}}.
\]
\end{lemma}
\begin{proof}
First count the number of ways of choosing the subsets that make up the cycles, and then multiply by the number of ways to arrange the elements of these subsets into cycles. The result is
\[
 \binom{n}{\underbrace{1 \cdots 1}_{c_1} \underbrace{2 \cdots 2}_{c_2} \cdots \underbrace{m\cdots m}_{c_m}} 
 \frac{1}{c_1! \cdots c_m!} \, \times \, \prod_{j=1}^m (j-1)!^{c_j},
\]
which simplifies to the claimed expression.
\end{proof}
 
Our next lemma is an analogue for permutations of a basic lemma from sieve theory.
 
 \begin{lemma}\label{sieve}
Suppose that $m,n$ are integers with $1 \leq m \leq n$. Let $\pi \in \cS_n$ be chosen uniformly at random. Then
\[ \frac{1}{2m} \leq \P(\mbox{$\pi$ has no cycle of length $< m$}) \leq \frac{1}{m}.\] \end{lemma}
\emph{Remarks.} Both upper and lower bounds are best possible, since trivially the probability in question is exactly $1/n$ when $n/2<m\le n$ (if a permutation has no cycle of length $< m$, with $m$ in this range, then it must be an $n$-cycle). In fact, it is not difficult to prove an asymptotic formula
 $ \sim \omega(n/m)/m$ ($n\to\infty$, $m\to\infty$, $m\le n$) for the probability in question, where $\omega$ is Buchstab's function
 and $\omega(u)\to e^{-\gamma}$ as $u\to\infty$ \cite[Theorem 2.2]{granville}.

 \begin{proof}
 (See the proof of \cite[Theorem 2.2]{granville}). We phrase the proof combinatorially rather than probabilistically; thus let $c(n,m)$ be the number of permutations of $\cS_n$ that have no cycles of length $<m$. We proceed by induction on $n$, the result being trivial when $n=1$.  Let 
 $\sum^*$ denote a sum over permutations with no cycle of length $<m$.  Using the fact that the sum 
 of lengths of cycles in a permutation in $\cS_n$ is $n$, we get
 \begin{align*}
  n c(n,m) &= {\sum_{\pi \in \cS_n}}^* n = {\sum_{\pi \in \cS_n}}^* \ssum{\sigma|\pi \\ \sigma\text{ a cycle}} |\sigma| = 
  \sum_{k\ge m} k \sum_{\sigma \in \cC_{k,n}} {\ssum{\pi \in \cS_n \\ \sigma | \pi}}^* 1 \\
  &= \sum_{m\le k\le n-m} k \sum_{\sigma \in \cC_{k,n}} c(n-k,m) + \sum_{\sigma\in \cC_{n,n}} n \\
  &= n! + \sum_{m\le k\le n-m} \frac{n!}{(n-k)!} c(n-k,m).
 \end{align*}
If $\frac{n}{2} <m \le n$, then $c(n,m)=\frac{n!}{n}$ and the result follows.  Otherwise, by the induction hypothesis,
\[
 n c(n,m) \le n! + \sum_{m\le k\le n-m} \frac{n!}{m} = n! \( 1 + \frac{n-2m+1}{m} \) \le \frac{n! \cdot n}{m}
\]
and
\[
 n c(n,m) \ge n! + \sum_{m\le k\le n-m} \frac{n!}{2m} = n! \( 1 + \frac{n-2m+1}{2m} \) \ge \frac{n! \cdot n}{2m}.
 \qedhere
\]
\end{proof}

It is now a simple matter to establish Proposition \ref{sieve2}. 

\begin{proof}[Proof of Proposition \ref{sieve2}]
 Let $t=c_1+2c_2+\cdots+mc_m$.  For each choice of the $c_1+\cdots+c_m$ disjoint cycles consisting of $c_j$ cycles from $\cC_{j,n}$ ($1\le j\le m$), there are $c(n-t,m+1)$ permutations $\pi\in \cS_n$ containing these cycles as factors and no other cycles of length at most $m$, where $c(n-t, m+1)$ is the number of permutations on $n-t$ letters with no cycle of length $< m+1$, as in the proof of Lemma \ref{sieve}. Applying Lemmas \ref{Cauchy} and \ref{sieve} completes the proof.
\end{proof}

\section{The local-to-global principle}\label{first-red}

As in the introduction, let $X_1, X_2, \dots$ be independent random variables with distribution $X_j \deq \Pois(1/j)$. We record here that
\begin{equation}\label{ELXk}
\E |\sL(\X_k)| = \sum_{c_1,\ldots,c_k\ge 0} |\sL(\c)| \P (X_1=c_1)\cdots \P (X_k=c_k) =
e^{-h_k} \sum_{c_1,\ldots,c_k\ge 0} \frac{|\sL(\c)|}{\prod_{i=1}^k c_i! i^{c_i}},
\end{equation}
where $h_k=1+\frac12 + \cdots + \frac{1}{k}$. We also record the inequalities
\begin{equation}\label{hk}
\log(k+1) \le h_k \le 1 + \log k, \qquad (k\geq 1)
\end{equation}
which may be proved, for example, by summing the obvious inequalities $\frac{1}{n+1} \le \int_n^{n+1} dt/t \le \frac{1}{n}$.

\begin{lemma}\label{Lclem}
Let $k\in \N$, $c_1,\ldots,c_k\ge 0$, $I\subset [k]$ and $c_i'=c_i$ for $i\not\in I$, $c'_i=0$ for $i\in I$.
then
$$|\sL(\c)| \le  |\sL(\c')| \prod_{i\in I} (c_i+1).$$
\end{lemma}

\begin{proof}
Clearly, $\sL(\c)$ is the union of $\prod_{i\in I} (c_i+1)$ translates of $\sL(\c')$. 
\end{proof}

\begin{lemma}\label{lem555}
Suppose that $\ell' \leq \ell$. Then
\[ \frac{1}{\ell}\E |\mathscr{L}(\X_{\ell})| \le \frac{1}{\ell'} \E |\mathscr{L}(\X_{\ell'})|.\]
\end{lemma}

\begin{proof}
By Lemma \ref{Lclem}, $|\sL(\X_{\ell})| \leq (1+X_{\ell' + 1})\cdots(1+X_{\ell})|\mathscr{L}(\X_{\ell'})|.$ Thus by independence,
\[ \E|\mathscr{L}(\X_{\ell})| \le \bigg( \prod_{i = \ell'+1}^{\ell} \E (1+X_i) \bigg)  \E |\sL(\X_{\ell'})|  = \frac{\ell+1}{\ell'+1} \E |\sL(\X_{\ell'})| \leq \frac{\ell}{\ell'}\E|\sL(\X_{\ell'})|. \qedhere
\]
\end{proof}

We also need to compute the mixed moments of $|\sL(\X_k)|$ with powers of some $X_j$. Recall that the $m$th moment $\E X^m$, if $X \deq \Pois(1)$, is the $m$th Bell number $B_m$. The sequence of Bell numbers starts $1,2,5,15, 52, 203,\dots$. 

\begin{lemma}\label{lem-sum-gen}
Suppose that $j_1,\dots, j_h \leq k$ are distinct integers and that $a_1,\dots, a_h$ are positive integers. Then
\[ \E |\mathscr{L}(\X_k)| X_{j_1}^{a_1} \cdots X_{j_h}^{a_h} \leq \frac{C_{a_1,\dots, a_h}}{j_1 \dots j_h} 
\E |\mathscr{L}(\X_k)|.\]
We may take $C_{a_1,\dots, a_h} = \prod_{i = 1}^h (B_{a_i}+B_{a_i+1})$. In particular we may take $C_1  = 3$.
\end{lemma}

\begin{proof}
Define $\X'_k$ by putting $X'_{j_1}=\cdots=X'_{j_h}=0$ and $X'_j=X_j$ for all other $j$. By Lemma~\ref{Lclem}, we have
\[ |\mathscr{L}(\X_k)| \leq |\mathscr{L}(\X'_k)|(1+X_{j_1})\cdots(1+X_{j_h}).\]
Thus by independence
\begin{equation}\label{eqmomentbound}
\E |\mathscr{L}(\X_k)|X_{j_1}^{a_1}\cdots X_{j_h}^{a_h} \leq \E|\mathscr{L}(\X'_k)| \prod_{i=1}^h (\E X_{j_i}^{a_i} + \E X_{j_i}^{a_{i+1}}).
\end{equation}
For $X \deq \Pois(\lambda)$ we have $\E X^m = \phi_m(\lambda)$, where $\phi_m(\lambda)$ is the $m$-th Touchard (or Bell) polynomial, a polynomial with positive coefficients and zero constant coefficient. If $\lambda \le 1$, it follows that $\E X^m \leq \lambda B_m$ for $m \geq 1$. The result follows immediately from this, \eqref{eqmomentbound}, and the observation that $ \E|\mathscr{L}(\X'_k)| \le \E|\mathscr{L}(\X_k)|$.
\end{proof}

We turn now to the proof of Proposition~\ref{first-reduction}. 
In what follows write $S(\X_\ell) = X_1 + 2X_2 + \dots + \ell X_{\ell} = 
\max \mathscr{L}(\X_\ell)$. 
We will treat the lower bound and upper bound in Proposition \ref{first-reduction} separately, 
the former being somewhat more straightforward than the latter. 

\begin{proof}[Proof of Proposition \ref{first-reduction} (Lower bound)] If $k<40$ then $i(n,k)\order i(\infty,k)\order 1\order \frac{1}{k}\E|\sL(\X_k)|$, so we may assume $k\geq 40$. Let $r=\fl{k/20}$ (so $r\ge 2$), and consider the permutations $\pi=\a \sigma_1 \sigma_2 \beta \in \cS_n$, where $\sigma_1$ and $\sigma_2$ are cycles, $|\a| \le 4r < |\sigma_1| < |\sigma_2| < 16r$, all cycles in $\a$ have length $\le r$,
all cycles in $\beta$ have length at least $16r$, and $\a \sigma_1 \sigma_2$ has a fixed set of size $k$. Because of the size restrictions on $\alpha,\sigma_1,\sigma_2$, if $\a$ is of type $\c=(c_1,\ldots,c_r)$, with $c_i$ cycles of length $i$ for $1\le i\le r$, then the last condition is equivalent to $k-|\sigma_1|- |\sigma_2| \in \sL(\c)$. In particular $|\sigma_1| +|\sigma_2|\le k$, and hence $n-|\a| - |\sigma_1|- |\sigma_2| \ge \frac45 k \ge 16r$. Fix $\c$ and $\ell_1, \ell_2$ with $4r<\ell_1<\ell_2<16r$ such that $k-\ell_1-\ell_2 \in \sL(\c)$. By Proposition \ref{sieve2}, the probability that a random $\pi \in \cS_n$ has $c_i$ cycles of length $i$ ($1\le i\le r$), one cycle each of length $\ell_1, \ell_2$ and no other cycles of length $<16r$ is at least
\[
 \frac{1}{32r \ell_1 \ell_2 \prod_{i=1}^r c_i! i^{c_i}} \ge \frac{1}{2^{13} r^3 \prod_{i=1}^r c_i! i^{c_i}}.
\]
For any $\ell_1$ satisfying $4r+1 \le \ell_1 \le 8r-1$, there are $|\sL(\c)|$ admissible values of $\ell_2>\ell_1$ for which $k-\ell_1-\ell_2 \in \sL(\c)$, since $\max \sL(\c) \le 4r \le k/5$.  We conclude that
\[
 i(n,k) \ge \frac{4r-1}{2^{13} r^3} \ssum{c_1,\cdots,c_{r}\ge 0 \\ S(\c) \le 4r} \frac{|\sL(\c)|}{\prod_{i=1}^r c_i! i^{c_i}}.
\]
As in \eqref{ELXk}, the sum above equals $e^{h_r}\E |\sL(\X_r)| 1_{S(\X_r)\le 4r}$.
Hence, by \eqref{hk}, we see that
\[
 i(n,k) \ge \frac{1}{2^{11} r}  \E |\sL(\X_r)| 1_{S(\X_r)\le 4r}.
\]
To estimate this, we use the inequality
\[ 1_{S(\X_r) \leq 4r} \geq 1 - \frac{S(\X_r)}{4r}.\] 
By Lemma \ref{lem-sum-gen} we have
\[ \E |\mathscr{L}(\X_r)| S(\X_r)  = \sum_{j = 1}^{r} \E |\mathscr{L}(\X_r)| j X_j  \leq 3r \E |\mathscr{L}(\X_r)|.\]
It follows that 
\[ i(n,k) \geq \frac{1}{2^{13} r} \E |\mathscr{L}(\X_r)|.\]
Finally, the lower bound in Proposition \ref{first-reduction} is a consequence of this and Lemma 
\ref{lem555}\footnote{Strictly for the purposes of proving our main theorem, this appeal to Lemma \ref{lem555} 
is unnecessary. However, that lemma is straightforward and it is more aesthetically pleasing to have 
$\E|\mathscr{L}(\X_k)|$ in the lower bound for $i(n,k)$ rather than $\E |\mathscr{L}(\X_r)|$.}.
\end{proof}

\begin{proof}[Proof of Proposition \ref{first-reduction} (Upper bound)]
Temporarily impose a total ordering on the set of all cycles $\bigcup_{k=1}^n \cC_{k,n}$, first ordering them by length, then imposing an arbitrary ordering of the cycles of a given length. Let $\pi\in \cS_n$ have an invariant set of size $k$.  Let $k_1=k$ and $k_2=n-k$. Then $\pi=\pi_1 \pi_2$, where $\pi_j$ is a product of cycles which, all together, have total length $k_j$, for $j=1,2$.  For some $j\in\{1,2\}$, the largest cycle in $\pi$, with respect to our total ordering, lies in $\pi_{3-j}$. Let $\sigma$ be the largest cycle in $\pi_j$, and note that $|\sigma|\leq\min(k_1,k_2)=k$. Write $\pi=\a \sigma \b$, where $\a$ is the product of all cycles dividing $\pi$ which are smaller than $\sigma$ and $\beta$ is the product of all cycles which are larger than $\sigma$. In particular $|\beta| \geq |\sigma|$ since $\beta$ contains the largest cycle in $\pi$ as a factor, and thus $|\sigma| \le |\beta| = n - |\sigma| - |\a|$.

By definition of $\sigma$ and $\alpha$, $\a \sigma$ has a divisor of size $k_j$. Suppose $|\sigma|=\ell$ and $\c=(c_1,c_2,\ldots,c_\ell)$ represents how many cycles $\a$ has of length $1,2,\ldots,\ell$, respectively. Then $k_j-\ell \in \sL(\c)$. For $\ell$ and $\c$ satisfying this last condition, the number of possible pairs $\a,\sigma$ is at most (by Lemma \ref{Cauchy})
\[
 \frac{n!}{(n-|\a|-|\sigma|)!}  \prod_{i<\ell} \frac{1}{c_i! i^{c_i}} \, \times \frac{1}{(c_\ell+1)! \ell^{c_{\ell}+1}}
\le  \frac{n!}{\ell (n-|\a|-|\sigma|)!} \prod_{i \le \ell} \frac{1}{c_i! i^{c_i}}.
\]
Given $\a$ and $\sigma$, since $|\sigma|\leq n-|\a|-|\sigma|$, Lemma~\ref{sieve} implies that the number of choices for $\beta$ is at most $(n-|\a|-|\sigma|)!/|\sigma|$. Thus
\[
 i(n,k) \le \sum_{j=1}^2 \sum_{\ell=1}^k \frac{1}{\ell^2} \ssum{c_1,\ldots,c_\ell \ge 0 \\ k_j-\ell \in
 \sL(\c)} \prod_{i \le \ell} \frac{1}{c_i! i^{c_i}} 
 = \sum_{j=1}^2 \sum_{c_1,\ldots,c_k \ge 0} \; \prod_{i \le k} \frac{1}{c_i! i^{c_i}} \ssum{m(\c) \le \ell
 \le k \\ k_j-\ell \in  \sL(\c)} \frac{1}{\ell^2},
\]
where $m(\c)=\max\{i:c_i>0\}\cup\{1\}$. With $\c$ fixed, note that $\ell \ge \max(m(\c),k_j-S(\c))$. Also, the number of $\ell$ such that $k_j-\ell\in \sL(\c)$ is at most $|\sL(\c)|$. Thus, the innermost sum on the right side above is at most
\[
 \frac{|\sL(\c)|}{\max(m(\c),k_j-S(\c))^2}.
\]
Like \eqref{ELXk}, using \eqref{hk} we thus see that
\begin{equation}\label{upper-1}
  i(n,k) \le 2ek \, \E \, \frac{|\sL(\X_k)|}{\max(m(\X_k),k-S(\X_k))^2}.
\end{equation}
To bound this we use the inequality
\[
 \frac{1}{\max(m,k-S)^2} \le \frac{4}{k^2} \(1 + \frac{S^2}{m^2}\),
\]
which can be checked in the cases $S\geq k/2$ and $S\leq k/2$ separately. It follows from this and \eqref{upper-1} that
\begin{equation}\label{upper-1a}
i(n,k) \leq 8e \frac1k \E |\sL(\X_k)| + 8e \frac1k \E\frac{|\sL(\X_k)|S(\X_k)^2}{m(\X_k)^2} .
\end{equation}

The first of these two terms is what we want, but the second requires a keener analysis. By conditioning on $m=m(\X_k)$ we have
\begin{align*}
 \E \frac{|\sL(\X_k)| S(\X_k)^2}{m(\X_k)^2} &= \sum_{m=1}^k \frac{1}{m^2}
 \ssum{c_1,\dots,c_m\ge 0 \\ c_m\ge 1} |\sL(\c)| S(\c)^2 \P (\X_m=\c) \P (X_{m+1}=\cdots=X_{k_j}=0) \\
&= \sum_{m=1}^{k_j} \frac{1}{m^2} \E \, \Y_m S(\X_m)^2 1_{X_m\ge 1} \exp\bigg( - \sum_{j=m+1}^k \frac{1}{j} \bigg) \\
& \leq \frac{e}k \sum_{m=1}^k \frac1m \E \, \Y_m S(\X_m)^2 X_m.
\end{align*}
Here we have written $\Y_m=|\sL(\X_m)|$ for brevity, and in the last step we used the crude inequality $1_{X_m\ge 1} \le X_m$. Expanding $S(\X_m)^2 = (X_1 + 2X_2 + \dots + mX_m)^2$ and using \eqref{upper-1a}, we arrive at
\begin{equation}\label{upper-1b}
 i(n,k) \ll \frac1k \E |\sL(\X_k)| + \frac1{k^2} \sum_{m=1}^k
 \frac{1}{m}  \sum_{i,i'=1}^m i i' \E\Y_m X_i X_{i'} X_m.
\end{equation}
The innermost sum is estimated using Lemma \ref{lem-sum-gen}, splitting into various cases depending
on the set of distinct values among $i,i',m$.
\begin{description}
	\item[Case 1] $i, i', m$ all distinct. Then $ii'\E\Y_m X_i X_{i'} X_m \leq \frac{C_{1,1,1}}{m}\E \Y_m = \frac{27}{m} \E \Y_m$.
	\item[Case 2] $i = i' \neq m$. Then $ii'\E\Y_m X_i X_{i'} X_m \leq \frac{C_{1,2}i}{m} \E \Y_m \leq C_{1,2}\E\Y_m = 21\E \Y_m$.
	\item[Case 3] $i = i' = m$. Then $ii'\E\Y_m X_i X_{i'} X_m \leq C_3 m \E \Y_m = 20 m \E \Y_m$.
	\item[Case 4] $i \neq i' =  m$ or $i' \neq i = m$. In both cases $ii'\E \Y_m X_i X_{i'} X_m \leq 21\E \Y_m$.
\end{description}
Summing over all cases, it follows that
\[
	\sum_{i,i'=1}^m ii'\E\Y_m X_i X_{i'} X_m \ll m \E \Y_m.
\]
Since clearly $\E \Y_m \le \E \Y_k$ for every $m\leq k$ the result follows from this and \eqref{upper-1b}.
\end{proof}

%
%
\section{The lower bound in Proposition~\ref{average-result}}\label{lower}
%
%

In this section we prove the lower bound in Proposition \ref{average-result}, and hence the lower bound in our
main theorem.   We begin by noting that from \eqref{ELXk} and \eqref{hk} follows
\begin{equation}\label{eq221}
 \E |\sL(\X_k)| \ge \frac{1}{ek} \sum_{c_1,\ldots,c_k\ge 0} \frac{|\sL(\c)|}{\prod_{i=1}^k c_i! i^{c_i}}.
\end{equation}
If we fix $r=c_1+\cdots+c_k$, which we may think of as the number of cycles in a random permutation, then
\begin{equation}\label{eq221a}
 \sum_{c_1+\cdots +c_k=r} \frac{|\sL(\c)|}{\prod_{i=1}^k c_i! i^{c_i}} = \frac{1}{r!} \sum_{a_1,\ldots,a_r=1}^k
 \frac{|\sL^*(\ba)|}{a_1\cdots a_r}, \end{equation}
 where
 \begin{equation}\label{sLstar} \sL^*(\ba)=\Big\{\sum_{i\in I} a_i : I \subset [r] \Big\}.
\end{equation}
The equality is most easily seen by starting from the right side and setting $c_i=|\{j:a_j=i\}|$ for each $i$: 
then $\sL(\c) = \sL^*(\ba)$, $\prod_{i = 1}^k i^{c_i} = a_1 \cdots a_k$, and each $\c = (c_1,\dots, c_k)$ comes from $\frac{r!}{c_1! \cdots c_k!}$ different choices of $a_1,\dots, a_k$. 
One may think of $a_1,\ldots,a_r$ as the (unordered) cycle lengths in a random permutation, in this case
conditioned so that there are $r$ total cycles.  

Now let $J= \fl{\frac{\log k}{\log 2}}$ and suppose that $b_1,\ldots,b_J$ are arbitrary non-negative integers with sum $r$. Consider the part of the sum in which
\[
 b_i = \sum_{j=2^{i-1}}^{2^i-1} c_j \;\; (i=1,2,\ldots,J), \qquad c_j=0 \; (j>2^J-1).
\]
Equivalently, suppose there are exactly $b_i$ of the $a_j$ in each interval $[2^{i-1},2^i-1]$. Writing $\DD(\bb) = \prod_{i=1}^J \{ 2^{i-1}, \ldots, 2^{i}-1 \}^{b_i}$, we have
\begin{equation}\label{eq222} 
\frac{1}{r!} \sum_{a_1,\ldots,a_r=1}^{2^J-1}
 \frac{|\sL^*(\ba)|}{a_1\cdots a_r} = \sum_{b_1,\dots, b_J} \frac{1}{b_1! \cdots b_J!} \sum_{\d \in \DD(\bb)} \frac{|\sL^*(\d)|}{d_1\cdots d_r}.\end{equation}
To see this, fix $b_1, \ldots, b_J$ and observe that there are $\frac{r!}{b_1! \cdots b_J!}$ ways to choose 
which $b_i$ of the variables $a_1,\ldots,a_r$ lie in $[2^{i-1},2^{i}-1]$ for $1\le i\le J$.
 
Combining \eqref{eq221}, \eqref{eq221a} and \eqref{eq222} gives
\[ \E |\sL(\X_k)| \gg \frac{1}{k} \sum_r \sum_{b_1+ \dots + b_J = r} \frac{1}{b_1! \cdots b_J!} \sum_{\d \in \DD(\bb)} \frac{|\sL^*(\d)|}{d_1 \cdots d_r}.  \] Thus in particular one has
\begin{equation}\label{eq223} \E |\sL(\X_k)| \gg \frac{1}{k}  \sum_{b_1+ \dots + b_J = J} \frac{1}{b_1! \cdots b_J!} \sum_{\d \in \DD(\bb)} \frac{|\sL^*(\d)|}{d_1 \cdots d_J}.  \end{equation}
(This may seem wasteful at first sight, but in fact a more careful -- though unnecessary -- analysis would reveal that the main contribution is from $r = J + O(1)$, so this is not in fact the case.) In the light of this, the motivation for proving the following lemma is clear. 

\begin{lemma}\label{sumd}
 For any $\bb=(b_1,\ldots,b_J)$ with $b_1 + \dots + b_J = J$ we have
 \[
 \sum_{\d \in \DD(\bb)} \frac{|\sL^*(\d)|}{d_1\cdots d_J} \gg \frac{(2\log 2)^{J}}{\sum_{i=1}^J 
 2^{b_1+\cdots+b_i-i}}.
 \]
\end{lemma}

\begin{proof}
Given $\ell \in \N$, let $R(\d, \ell)$ be the number of $I \subset [J]$ with $\ell = \sum_{i \in I} d_i$ 
(One should think of the number of cycles with lengths summing to precisely $\ell$ in a random permutation.) Then 
$\sum_{\ell} R(\d,\ell) = 2^{J}.$  Also, define $\lambda_i=\sum_{j=2^{i-1}}^{2^i-1} 1/j$ for $1\le i\le J$ (thus $\lambda_i \approx \log 2$).
By Cauchy-Schwarz,
\begin{equation}\label{CauchyD}
\begin{split}
 2^{2J} \prod_{j=1}^J \lambda_j^{2b_j} &= \bigg( \sum_{\d\in\DD(\bb)} \frac{1}{d_1\cdots d_J} \sum_\ell 
 R(\d,\ell) \bigg)^2 \\
 &=\bigg( \sum_{\d\in\DD(\bb)} \frac{1}{d_1\cdots d_J} \sum_{\ell\in \sL^*(\d)}  R(\d,\ell) \bigg)^2 \\
 &\le \bigg( \sum_{\d\in\DD(\bb), \ell} \frac{R(\d,\ell)^2}{d_1\cdots d_J} \bigg)
 \bigg( \sum_{\d\in\DD(\bb)} \frac{|\sL^*(\d)|}{d_1\cdots d_J} \bigg).
\end{split}
\end{equation}
Our next aim is to establish an upper bound for the first sum on the right side.  We have
\begin{equation}\label{RS1}
\sum_{\d\in\DD(\bb), \ell} \frac{R(\d,\ell)^2}{d_1\cdots d_J} = \sum_{I_1, I_2 \subset [J]} S(I_1,I_2),
\end{equation}
 where 
 
 \[ S(I_1,I_2) = \sum_{\substack{\d \in \DD(\bb) \\ \sum_{i\in I_1} d_i = \sum_{i\in I_2} d_i. } } \frac{1}{d_1 \cdots d_J}\]

 If $I_1=I_2$, then evidently  $S(I_1,I_2) = \lambda_1^{b_1} \cdots \lambda_J^{b_J}$.  
If $I_1$ and $I_2$ are distinct, let $j=\max(I_1\triangle I_2)$ be the largest coordinate at 
which $I_1$ and $I_2$ differ.  With all of the quantities $d_i$ fixed except for $d_j$, 
we see that $d_j$ is uniquely determined by the relation $\sum_{i\in I_1} d_i = \sum_{i\in I_2} d_i$.  If we define $e(j)\in [J]$ uniquely by
\[
 b_1 + \cdots + b_{e(j)-1}+1 \le j \le b_1 + \cdots + b_{e(j)},
\]
then $d_j \ge 2^{e(j)-1}$, regardless of the choice of $d_1,\dots, d_{j-1}, d_{j+1},\dots, d_J$ and thus
\[
 S(I_1,I_2) \leq \prod_{\substack{i = 1 \\ i \neq j}}^J(\sum_{d_i} \frac{1}{d_i}) \cdot \frac{1}{2^{e(j) - 1}} = \frac{\lambda_1^{b_1} \cdots \lambda_J^{b_J} \lambda_{e(j)}^{-1}}{2^{e(j)-1}} \ll \frac{\lambda_1^{b_1} \cdots \lambda_J^{b_J}}{2^{e(j)}}.
\] (Here, the sums over $d_i$ are over the appropriate dyadic intervals required so that $\d \in \DD(\bb)$.)
Here we used the fact that $\lambda_i \asymp 1$; in fact one may note that $\lambda_{i} \geq \lambda_{i+1}$ for all $i$ (since $\frac{1}{n} \geq \frac{1}{2n} + 
\frac{1}{2n+1}$) and that $\lim_{i \rightarrow \infty} \lambda_i = \log 2$, so in fact $\lambda_i \geq \log 2$
for all $i$.

Since the number of pairs of subsets $I_1,I_2\subset[J]$ with $\max(I_1\triangle I_2)=j$ is exactly 
$2^{J+j-1}$,  we get from this and \eqref{RS1} that
\begin{align*}
\prod_{j=1}^J \lambda_j^{-b_j} \sum_{d\in\DD(\bb), \ell} \frac{R(\d,\ell)^2}{d_1\cdots d_J} &\ll
2^J + 2^J \sum_{j=1}^J 2^{j-e(j)} \
	= 2^J + 2^J \sum_{i=1}^J 2^{-i} \sum_{j:e(j)=i} 2^j\\
	&\ll 2^J + 2^J \sum_{i=1}^J 2^{b_1+\cdots+b_i-i}\\
	&\ll 2^J \sum_{i=1}^J 2^{b_1+\cdots+b_i-i}.
\end{align*}
Comparing with \eqref{CauchyD}, and using again that  $\lambda_i \geq \log 2$, completes the proof.
\end{proof}

Combining Lemma \ref{sumd} and \eqref{eq223}, we obtain
\begin{equation}\label{eqbeforecyclelemma}
\E |\mathscr{L}(\X_k)| \gg \frac{(2\log 2)^J}{k} \sum_{b_1+\cdots+b_J=J} \frac{1}{b_1! \cdots b_J! 
\sum_{i=1}^J 2^{b_1+\cdots+ b_i-i}}.
\end{equation}
Somewhat surprisingly, the right hand side here can be evaluated explicitly using the ``cycle lemma'', as in~\cite{ford-2}. The key trick is to add an additional averaging over the $J$ cyclic permutations of $b_1,\ldots,b_J$ to the inner summation.

\begin{lemma}\label{cyclelemma}
Let $x_1,\dots,x_J$ be positive reals such that $x_1\cdots x_J = 1$. Then the average of $\big(\sum_{i=1}^J x_1\cdots x_i\big)^{-1}$ over cyclic permutations of $x_1,\dots,x_J$ is exactly $1/J$.
\end{lemma}
\begin{proof}
Reading indices modulo $J$ we have
\[
\sum_{t=1}^J \frac{1}{\sum_{i=1}^J x_{t+1}\cdots x_{t+i}} = \sum_{t=1}^J \frac{x_1\cdots x_t}{\sum_{i=1}^J x_1\cdots x_{t+i}} = 1.\qedhere
\]
\end{proof}

Applying the cycle lemma with $x_i = 2^{b_i - 1}$ gives (noting that cyclic permutation of the variables is a 1-1 map on the set of $(b_1,\dots, b_J)$ with $b_1 + \dots + b_J = J$) that
\[ \sum_{b_1+\cdots+b_J=J} \frac{1}{b_1! \cdots b_J! 
\sum_{i=1}^J 2^{b_1+\cdots+ b_i-i}} = \frac{1}{J} \sum_{b_1+\cdots+b_J=J} \frac{1}{b_1!\cdots b_J!} = \frac{1}{J} \cdot \frac{J^J}{J!},\] the second equality being a consequence of the multinomial theorem.

Substituting into \eqref{eqbeforecyclelemma}, and recalling that $J=\frac{\log k}{\log 2} + O(1)$, the lower bound in Proposition~\ref{average-result} now follows from Stirling's formula.

\section{The upper bound in Proposition~\ref{average-result}}\label{upper}

In this section we turn to the upper bound in Proposition \ref{average-result}, that is to say the bound
\[ \E |\mathscr{L}(\X_k)| \ll k^{\alpha} (\log k)^{-3/2}.\]
As with the lower bound, we condition on the number of cycles of length at most $k$ in a random permutation.
Recall from \eqref{sLstar} the definition of $\sL^*(\ba)$:
\[ \sL^*(\ba) = \Big\{\sum_{i\in I} a_i : I \subset [r] \Big\}.\]
From \eqref{ELXk}, \eqref{hk} and \eqref{eq221a} we have
\begin{equation}\label{upper-start}
\E |\sL(\X_k)| \le \frac{1}{k} \sum_r \frac{1}{r!}  \sum_{a_1,\ldots,a_r=1}^k
 \frac{|\sL^*(\ba)|}{a_1\cdots a_r}.
\end{equation}
The most common way for $|\sL^*(\ba)|$ to be small is when there are many of the $a_i$ which are small.
To capture this, let $\tilde{a}_1, \tilde{a}_2,\ldots$ be the increasing rearrangement of the sequence
$\ba$, so that $\tilde{a}_1 \le \tilde{a}_2 \le \cdots$.  For any $j$ satisfying $0\le j\le r$, we have
\[
 \sL^*(\ba) \subset \left\{ m + \sum_{i\in I} \tilde{a}_i : 0\le m\le \sum_{i=1}^j \tilde{a}_i, \, I \subset \{j+1, \ldots, r\} \right\},
\]
from which it follows immediately that
\[
|\sL^*(\ba)| \le G(\ba),
\]
where
\begin{equation}\label{LG}
G(\ba) = \min_{0\le j\le r} 2^{r-j} \( \tilde{a}_1 + \cdots + \tilde{a}_j + 1 \).
\end{equation}
It is reasonable to expect that
\begin{equation}\label{up-approx}
  \sum_{a_1,\ldots,a_r=1}^k \frac{G(\ba)}{a_1\cdots a_r} \sim  \int^k_1\!\! \cdots \!\!\int^k_1 \frac{G(\mathbf{t})}{t_1\cdots t_r} d\mathbf{t}
  = (\log k)^r  \int^1_0 \!\!\cdots\!\! \int^1_0 G(e^{\xi_1\log k},\ldots,e^{\xi_r\log k}) d\boldsymbol{\xi},
\end{equation}
where we have enlarged the domain of $G$ to include $r$-tuples of positive real numbers.  However, $G$ is not an especially regular function and so \eqref{up-approx} is perhaps too much to hope for.  The function $G$ is, however, increasing
in every coordinate and we may exploit this to prove an approximate version of \eqref{up-approx}.

\begin{lemma}\label{lemma4.1}
For any $r\ge 1$, we have
\[
\sum_{a_1,\ldots,a_r=1}^k \frac{|\sL^*(\ba)|}{a_1\cdots a_r} \ll (2h_k)^r r! \int_{\Omega_r}
\min_{0\le j\le r} 2^{-j} (k^{\xi_1} + \cdots + k^{\xi_j} + 1 ) d\boldsymbol{\xi},
\]
where $\Omega_r = \{ (\xi_1, \dots, \xi_r) : 0 \leq \xi_1 \leq \xi_2 \leq \dots \leq \xi_r \leq 1\}$.
\end{lemma}

\begin{proof}
Motivated by the fact that $1/a = \int_{\exp(h_{a-1})}^{\exp(h_a)} dt/t$, define the product sets
\[
 R(\ba) = \prod_{i=1}^r \left[ \exp\( h_{a_i-1} \), \exp\( h_{a_i} \) \right].
\]
By \eqref{LG}, we have
\[
  \sum_{a_1,\ldots,a_r=1}^k \frac{|\sL^*(\ba)|}{a_1\cdots a_r} \le \sum_{a_1,\ldots,a_r=1}^k \frac{G(\ba)}{a_1\cdots a_r}
  =  \sum_{a_1,\ldots,a_r=1}^k G(\ba) \int_{R(\ba)} \frac{d \mathbf{t}}{t_1\cdots t_r}.
\]
Consider some $\mathbf{t} \in R(\ba)$.  Writing $\tilde t_1 \leq \tilde t_2 \leq \dots \leq \tilde t_r$ for the non-decreasing rearrangement of $\mathbf{t}$, we have
\[
  \exp\( h_{\tilde{a}_i-1} \) \le \tilde{t_i} \le  \exp\( h_{\tilde{a}_i} \) \quad \mbox{for $1\le i\le r$}. 
\]
From \eqref{hk} we see that $\tilde{t}_i \ge \tilde{a}_i$ for all $i$. Hence
\[
 G(\ba) \le \min_{0\le j\le r} 2^{r-j} (\tilde{t_1} + \cdots + \tilde{t_j}+1) = G(\mathbf{t}) \quad \mbox{for all $\mathbf{t}\in R(\ba)$}.
\]
This yields
 \[
 \sum_{a_1,\ldots,a_r=1}^k G(\ba) \int_{R(\ba)} \frac{d \mathbf{t}}{t_1\cdots t_r} \le
  \sum_{a_1,\ldots,a_r=1}^k \int_{R(\ba)} \frac{G(\mathbf{t})}{t_1\cdots t_r} d \mathbf{t} = \int_1^{\exp(h_k)}\cdots \int_1^{\exp(h_k)}
  \frac{G(\mathbf{t})}{t_1\cdots t_r} d \mathbf{t}.
 \]
The integrand on the right is symmetric in $t_1,\ldots,t_r$.  Making the change of variables $t_i=e^{\xi_i h_k}$
yields
\[
  \sum_{a_1,\ldots,a_r=1}^k \frac{|\sL^*(\ba)|}{a_1\cdots a_r} \le (2h_k)^r r! \int_{\Omega_r} 
  \min_{0\le j\le r} 2^{-j}\( e^{\xi_1 h_k} + \cdots + e^{\xi_j h_k} + 1 \) d \boldsymbol{\xi}.
\]
The lemma follows from the upper bound in \eqref{hk}, namely $h_k\le 1+ \log k$.
 \end{proof}

With Lemma \ref{lemma4.1} established, we may conclude the proof of the upper bound in Proposition \ref{average-result} by quoting \cite[Lemma 3.6]{ford-2}. Indeed, in the notation of that paper
\[ \int_{\Omega_r} \min_{0\le j\le r} 2^{-j}\( k^{\xi_1} + \cdots + k^{\xi_j} + 1 \) d \boldsymbol{\xi}  = U_r(\log_2 k),\] and thus by \eqref{upper-start} and  Lemma \ref{lemma4.1} we have
\begin{equation}\label{lkx} \E |\mathscr{L}(\X_k)|  \ll \frac{1}{k} \sum_r (2 h_k)^r U_r(\log_2 k).\end{equation}
Now \cite[Lemma 3.6]{ford-2} provides the bound
\[ U_r(\log_2 k) \ll \frac{1 + |\log_2 k - r|^2}{(r+1)! (2^{r - \log_2 k} + 1)},\] uniformly for $0 \leq r \leq 10 \log_2 k$ .
Set 
\[ r_* = \lfloor \log_2 k\rfloor.\]
In what follows, we will use the observation that $a^n/(n+1)!$ is increasing for $n \leq a - 2$ and decreasing thereafter. 
If $r = r_*+ m$ with $m \leq 9 \log_2 k$, $m \in \Z_{\geq 0}$, then we have
\begin{align*} (2h_k)^r U_r(\log_2 k) & \ll \frac{(\frac{4}{3}h_k)^r}{(r+1)!}  \cdot \pfrac{3}{2}^r \cdot \frac{1 + m^2}{2^m} \\ & \ll \frac{(\frac{4}{3}h_k)^{r_*}}{(r_*+1)!}  \cdot \pfrac{3}{2}^{r_*} \cdot \frac{1 + m^2}{(\frac{4}{3})^m} \\ & \ll k^{1 + \frac{1 + \log\log 2}{\log 2}} (\log k)^{-3/2} \cdot  \frac{1 + m^2}{(\frac{4}{3})^m}.\end{align*}
In the first step we used the observation (and the fact that $\frac{4}{3} < \frac{1}{\log 2}$), and in the second step we used Stirling's formula and \eqref{hk}. 
Summed over $m$, this is of course rapidly convergent and shows that the contribution to \eqref{lkx} from this range of $r$ is acceptable. 

Next suppose that $r = r_* - m$, $m \in \N$. Then we have
\begin{align*}
(2h_k)^r U_r(\log_2 k) & \ll \frac{(\frac{3}{2}h_k)^r}{(r+1)!}  \cdot \pfrac{4}{3}^r \cdot (1 + m^2) \\ & \ll \frac{(\frac{3}{2}h_{k})^{r_*}}{(r_*+1)!} \cdot \pfrac{4}{3}^r \cdot (1 + m^2) \\ & \ll k^{1 + \frac{1 + \log\log 2}{\log 2}} (\log k)^{-3/2} \cdot  \frac{1 + m^2}{(\frac{4}{3})^m}.
\end{align*}
Here, we used the observation (and the fact that $\frac{3}{2} > \frac{1}{\log 2}$) and a second application of Stirling's formula.  
Summed over $m$, this is once again rapidly convergent and the contribution to \eqref{lkx} from this range of $r$ is acceptable. 

There remains the range $r > 10 \log_2 k$. Here, we use the trivial bound
$U_r(\log_2 k) \leq 1/r!$  and thus
\[ \sum_{r > 10 \log_2 k} (2h_k)^r U_r(\log_2 k)  \ll \sum_{r > 10 \log_2 k}  \frac{(2h_k)^r}{r!}  \ll k^{-10},
\] which is obviously minuscule in comparison to the other terms. 

\emph{Remarks.} It is obvious from this analysis and the lower bound in our main theorem that a proportion $\geq 1 - \eps$ 
of all permutations fixing
some set of size $k$ have $\log_2 k + O(\log (1/\eps))$ cycles of length at most $k$. 
It is most probably also true that for a proportion $\geq 1-\eps$ of all permutations fixing some set of size
$k$ we have $\log \tilde a_j \geq j \log 2 - O_{\eps}(1)$ 
for $j\le \log_2 k - O_\eps(1)$, where the $\tilde a_j$ are the (ordered) cycle lengths of the permutation. 
To establish this would require opening up some of the arguments used to bound the quantities $U_k$ in \cite{ford-2}. We plan to return to this and other issues in a future paper.

\emph{Acknolwedgments}. BG is supported by ERC Starting Grant number 279438, \emph{Approximate algebraic structure and applications}, and a Simons Investigator Award. KF is supported by National Science Foundation grants DMS-1201442 and DMS-1501982.

\bibliographystyle{alpha}
\bibliography{fixed-point-permutation}

\end{document}